\documentclass{amsart}
\usepackage[utf8]{inputenc}

\usepackage{amsmath,mathtools,amsfonts,amssymb,amsthm}
\usepackage{graphicx}
\usepackage{float}
\usepackage{color}
\usepackage{hyperref}
\usepackage[shortlabels]{enumitem}
\usepackage{enumitem}
\usepackage{geometry} 
\hypersetup{colorlinks,linkcolor={blue},citecolor={red},urlcolor={blue}}

\newtheorem{thm}{Theorem}[section]
\newtheorem{corol}[thm]{Corollary}
\newtheorem{lem}[thm]{Lemma}

\newtheorem{prop}[thm]{Proposition}
\newtheorem{dfn}[thm]{Definition}
\newtheorem{rmk}[thm]{Remark}

\DeclareMathOperator{\dist}{dist}

\DeclareMathOperator{\Ric}{Ric}

\DeclareMathOperator{\interior}{int}

\numberwithin{equation}{section}

\title{On the concentration for a singularly perturbed problem with nonlinear Neumann boundary condition}
\author{Eduardo Hitomi}
\address{IMECC - University of Campinas, 13083-859, Campinas - SP, Brazil}
\email{ehitomi@ime.unicamp.br}
\date{2018}
\thanks{This study was financed in part by the Coordena\c c\~ao de Aperfei\c coamento de Pessoal de N\'ivel Superior - Brazil (CAPES) - Finance Code $001$}

\begin{document}
\calclayout

\begin{abstract}
In this work, we prove the existence of a family of solutions of the Allen-Cahn equation with nonlinear Neumann boundary condition under some constraints, whose nodal sets concentrate asymptotically to a given volume nondegenerate capillary hypersurface in a compact Riemannian manifold. Our construction is inspired
by the work of  Pacard in \cite{pacard} and of Pacard and Ritor\'e in \cite{pacard-ritore}.
\end{abstract}

\maketitle

\pagestyle{plain}

\section{Introduction}
One of the main approaches to understand the nature of a Riemannian submanifold and its relations with the ambient space is to look at local foliation structures. It provides explicitly  neighborhoods of the submanifold so that one can perform an asymptotic study from the geometric properties of the leaves. Classical examples are the constructions of Ye in \cite{ye} around nondegenerate critical points of the scalar curvature by topological spheres, of Huisken and Yau in \cite{huisken-yau} around asymptotically flat $3$-manifolds near infinity by stable closed constant mean curvature (CMC) surfaces and of Mahmoudi, Mazzeo and Pacard in \cite{mahmoudi-mazzeo-pacard} around nondegenerate minimal submanifolds by CMC tubes. Recently, Li in \cite {li} followed the construction of Ambrozio in \cite{ambrozio} to provide a foliation by capillary surfaces around corners of polyhedra in a $3$-manifold and answer affirmatively the dihedral rigidity conjecture of Gromov.

One perspective to construct foliations is to consider the classical relation between the gradient theory of phase transitions and the class of submanifolds arising from geometric variational problems. In fact, given the works of Modica, Sternberg, Anzellotti et al. showing that, in a bounded domain $\Omega \in \mathbb{R}^{n},$ a sequence of minimizers of the Allen-Cahn energy
$$
E_{\varepsilon}(u) \coloneqq  \int_{\Omega} \dfrac{\varepsilon}{2} |\nabla u|^{2} + \dfrac{W(u)}{\varepsilon},
$$
where $W$ is a double-well potential (for instance, the typical  $W(s) = \frac{(1  - s^{2})^{2}}{4}$) under the mass constraint
$$
\int_{\Omega} u = m,
$$
for some fixed $m \in (-|\Omega|,|\Omega|),$ converges to $\pm 1$ in compact subsets so that the nodal sets concentrates around a hypersurface which minimizes area, Pacard and Ritor\'e in \cite{pacard-ritore} considered the converse question to show the existence of a foliation around any nondegenerate CMC hypersurface $\Sigma$ in a compact Riemannian manifold $M$ with or without boundary by nodal sets of critical points of the associated Allen-Cahn energies. In the case of nonempty boundary $\partial M,$ they assume that $\Sigma$ meets $\partial M$ orthogonally, that is, it is a free boundary one. Considering the great deal of work devoted to understand CMC hypersurfaces intersecting the boundary more generally with a fixed angle $\theta,$ that is capillary hypersurfaces, and the work of Modica \cite{modica}, we are interested in this paper the Pacard-Ritor\'e's analogous to this case.  For a broad introduction to capillar surfaces, see the book \cite{finn} of Finn.  Physically in $\mathbb{R}^{3},$ these surfaces correspond to the description of an incompressible liquid in a container in the absense of gravity and as an interface between two incompressible fluids. 
We also mention that recently, Wang and Xia in \cite{wang-xia} closed the classification problem in a ball, showing that any immersed stable capillary hypersurfaces in a ball in space forms are totally umbilical, generalizing the work \cite{nunes} of Nunes dealing with the minimal free boundary case.

Let us state the setting of this work. Let $(M,g)$ to be a $n$-dimensional compact Riemannian manifold with smooth boundary $\partial M \neq \emptyset$ and assume without loss of generality that $M$ is a subdomain of a larger Riemannian manifold $(\Tilde{M},\Tilde{g})$ with $\Tilde{g}|_{M} = g$ so that $\partial M$ is a smooth hypersurface of $\Tilde{M}.$  Let $W \colon \mathbb{R} \to \mathbb{R}$ be a smooth function which is positive away from $u = \pm 1.$ We assume that the points $u = \pm 1$ are the only minima, nondegenerate critical points of $W,$ and 
\begin{equation} 
\lim_{u \to 1^{-}} \dfrac{W'(u)}{\sqrt{W(u)}}\;\;\text{and}\;\; \lim_{u \to - 1^{+}} \dfrac{W'(u)}{\sqrt{W(u)}} \;\;\text{both exist}.
\end{equation}

For any $\varepsilon > 0$ and any function $u \colon M \to \mathbb{R},$ such that $u \in H^1(M),$ define the energy
\begin{equation} \label{eq.2.3}
E_{\varepsilon}(u) \coloneqq \int_{M} \dfrac{\varepsilon}{2} |\nabla u|_{g}^2 + \dfrac{W(u)}{\varepsilon}\;dv_{g} + \int_{\partial M} \sigma(u)\;d\mathcal{H}^{n-1},
\end{equation}
where $\sigma \colon \mathbb{R} \to \mathbb{R}$ is a function satisfying 
\begin{equation} \label{eq.sigma}
\sigma'(t) = \cos{\theta} \sqrt{2 W(t)},
\end{equation} 
with $\theta \in (0,\pi/2]$ fixed, without loss of generality, $\nabla$ denotes the gradient, $dv_{g}$ is the volume form on $M$ associated to the Riemannian metric $g$ and $\mathcal{H}^{n-1}$ is the $(n-1)$-dimensional Hausdorff measure. We agree that $E_{\varepsilon}(u) = \infty$ when $W(u) \notin L^1(M).$ We also define the mass constraint 
\begin{equation} 
V(u) = \int_{M} u\;dv_{g}.
\end{equation}

Given $c_{0} \in (-1,1),$  we can consider critical points of the energy $u \mapsto E_{\varepsilon}(u)$ under the constraint $V(u) = c_{0}|M|,$ where $|M|$ denotes the volume of $M.$ Such a critical point $u$ is solution of
\begin{equation} \label{eq.2.7}
-\varepsilon^2 \Delta_{g} u +  W'(u) = \varepsilon \lambda,
\end{equation}
in $M,$ where $\varepsilon \lambda \in \mathbb{R}$ corresponds to the Lagrange multiplier associated to the constraint $V(u) = c_{0}|M|.$ Moreover, $u$ satisfies the nonlinear Neumann boundary condition
\begin{equation} 
-\varepsilon g(\nabla u, \nu_{\partial M})= \sigma'(u),
\end{equation}
on $\partial M,$ where $\nu_{\partial M}$ denotes the unit outer normal vector field of $\partial M.$ The energy $E_{\varepsilon}$ corresponds to the typical total energy of a fluid within the van der Waals-Cahn-Hilliard theory of phase transitions, modelling separation phenomena for capillary surfaces. The Lagrange multiplier $\varepsilon \lambda,$ which appears in \eqref{eq.2.7}, is known in the physics literature as the chemical potential of the density configuration $u.$ Considering that in a bounded domain $\Omega$ of $\mathbb{R}^n,$ Modica in \cite{modica} showed the existence of energy minimizers $\{u_{\varepsilon}\}_{\varepsilon \in (0,1)}$ and convergence to a function $u$ in $L^{1}$ as $\varepsilon \to 0,$ with $u = \pm 1$ a.e. in $\Omega,$ and under a suitable assumption on $\sigma$ that the contact angle $\theta
$ formed by the boundary $\partial \Omega$ and the reduced boundary of $\{u = 1\}$ in $\Omega$ is equal to
\begin{equation} 
\theta = \arccos \left(\dfrac{\sigma(1) - \sigma(-1)}{c_{0}} \right),
\end{equation}
where
\begin{equation} 
c_{0} = \int_{-1}^{1} \sqrt{2W(s)}\;ds.
\end{equation}
We would like to address is the following question.

\textit{Assume that $\Sigma \subset M$ is a constant mean curvature hypersurface meeting $\partial M$ with a fixed angle $\theta \in (0,\pi/2].$ Does $\Sigma$ appear as the limit, as the parameter $\varepsilon$ tends to $0,$ of the nodal sets of a sequence of critical points of $E_{\varepsilon}$ subject to the constraint $V(u) = c_{0}|M|$?} 


The characterization of the contact angle condition is through the energy minimality of the $\Gamma$-limit functional and it is essential that $u_{\varepsilon}$'s are global energy minimizers for the $\Gamma$-convergence argument. In \cite{kagaya-tonegawa, kagaya-tonegawa1}, Kagaya and Tonegawa considered the more general case of critical points, analyzing the the contact angle condition due the presence of the boundary term in \eqref{eq.2.3} as $\varepsilon \to 0^{+}.$ However, the phase transitions approach via min-max methods for the existence theory of general capillary hypersurfaces in a compact manifold with boundary has not been developed yet, as done for instance in \cite{guaraco, gaspar-guaraco} in the setting of minimal hypersurfaces on closed manifolds.
The difficulties here arise in the boundary behavior which can be \textit{a priori} completely different from the interior as pointed in \cite{mizuno-tonegawa} by Mizuno and Tonegawa. Not only the location of the limit interface does not need to be the continous with respect to the boundary, that is the limiting solution on the boundary (see Theorem $3.2$ of \cite{kagaya-tonegawa1}) is not necessarily the trace of the limiting solution in the interior, but also the nodal sets can concentrate along the boundary, a wet dropping effect. 

Also, we note that concentration in Pacard-Ritor\'e's construction, and so in this work, happens with multiplicity $1.$ We can expect also that the concentration, not necessarily a foliation (the distribution occuring  $O(\varepsilon \log(\varepsilon))$ distant from each other), with arbitrary multiplicity holds as an analogue to the work \cite{delpino-kowalczyk-wei-yang} of Del Pino, Kowalczyk, Wei and Yang under the hypothesis of positive Ricci curvature case, at least for the free boundary case, in view of \cite{malchiodi-ni-wei}.

\subsection{Concentration on capillary  hypersurfaces}
Let $\Sigma^{n-1}$ a smooth hypersurface of $M$ with boundary $\partial \Sigma$ dividing $M$ into two components $M^{\pm}(\Sigma)$ such that $\partial \Sigma \subset \partial M,$ and $\Sigma$ meets $M$ with constant angle $\theta.$ Along this work, we denote by
\begin{itemize}
\item $\nu_{\partial_{M}}$ the unit normal vector field along $\partial M$ pointing outside $\partial M,$
\item $\tau_{\partial M}$ the unit normal vector field along $\partial \Sigma$ tangent to $\partial M,$ pointing into $M^{+}(\Sigma),$
\item $II_{\partial M}$ the second fundamental form on $\partial M$ with the sign convention $II_{\partial M}(X,Y) = g(\nabla_{X} \nu_{\partial M}, Y),$ for all $X,Y \in T\partial M,$
\item $\nu_{\Sigma}$ the local unit normal vector field of $\Sigma,$
\item $\tau_{\partial \Sigma}$ the unit conormal of $\partial \Sigma$ that points outside $\Sigma,$
\item $II_{\Sigma}$ the second fundamental form on $\Sigma$ with the sign convention $II_{\Sigma}(X,Y) = -g(\nabla_{X} \nu_{\Sigma}, Y),$ for all $X,Y \in T\Sigma,$
\item $H$ the mean curvature of $\Sigma.$
\end{itemize}

Since $\Sigma \subset M$ separates $M$ into two regions $M^{\pm}(\Sigma)$, it is the nodal set of some $u$ defined in $M$ and if $0$ is a regular value of $u,$ then $M \setminus \Sigma$ is the union of
$$
M^+(\Sigma) = u^{-1}((0,\infty))\;\;\;\text{and}\;\;\;M^+(\Sigma) = u^{-1}((-\infty,0)).
$$
Then by $\Sigma$ meeting $\partial M$ with constant angle $\theta,$ we mean that for every $p \in \partial \Sigma,$ we have 
\begin{equation} \label{eq.admiss}
g\left(\nabla u (p), \nu_{\partial M}(p)\right)= |\nabla u (p)| \cos(\theta),
\end{equation}

\begin{dfn} 
A smooth embedded hypersurface $\Sigma \subset M$ is $\theta$-admissible if $\Sigma$ is the nodal set of a smooth function $u$ for which $0$ is a regular value of $u$ and which satisfies the boundary condition \eqref{eq.admiss} for all $p \in \partial \Sigma.$
\end{dfn}

\subsubsection{Jacobi fields on a capillary hypersurface}
Let $\Omega \subset \partial M$ be one of the component $\partial M^{+} \cap \partial M.$ Let $\Phi_{t}$ be an admissible variation of $\Sigma$ with variation vector field $Y(p) = \frac{\partial \Phi_{t}}{\partial t}(p) \vert_{t = 0}$ for every $p \in \Sigma,$ that is $\Phi_{t}(\interior \Sigma) \subset \interior M$ and $\Phi_{t}(\partial \Sigma) \subset \partial M.$ If $A(t)$ denotes the volume of $\Sigma_{t} \coloneqq \Phi_{t}(\Sigma),$  $T(t)$ denotes the volume of $\Omega_{t} \coloneqq \Phi_{t}(\Omega),$ and  $V(t)$ denotes the volume enclosed by $\Sigma$ and $\Sigma_{t},$ consider the the total energy
$$
E(t) \coloneqq A(t) - \cos \theta T(t). 
$$
A variation is volume-preserving if $V(t) = V(0)$ for every $t.$ $\Sigma$ is capillary if and only if it is stationary for $E$ for any volume-preserving admissible variation, that is $E'(0) = V'(0) = 0,$ where
$$
E'(0) = -\int_{\Sigma} n H_{\Sigma} \left\langle Y,\nu_{\Sigma} \right\rangle\;da_{g} + \int_{\partial \Sigma} \left\langle Y, \tau_{\Sigma} - \cos \theta  \tau_{\partial M} \right\rangle\;ds_{g},
$$
and
$$ 
V'(0) = \int_{\Sigma} \left\langle Y, \nu_{\Sigma} \right\rangle\;da_{g}.
$$
Here $da_{g}$ and $ds_{g}$ are the volume forms on $\Sigma$ and $\partial \Sigma$ which are induced by the metric $g.$ This of course implies that $\Sigma$ is capillary if it has constant mean curvature and intersects $\partial M$ with constant angle $\theta,$ in the sense that the angle between $\tau_{\Sigma}$ and $\tau_{\partial M}$, or equivalently between $\nu_{\Sigma}$ and $\nu_{\partial M},$ is $\theta.$

Recall that the linearized operator of the mean curvature about $\Sigma$ is given by the second variation of $E,$ which for any volume-preserving admissible variation is
$$
E''(0) = -\int_{\Sigma} \omega J_{\Sigma}\omega\;da_{g} + \int_{\partial \Sigma} \omega \; B_{\Sigma}\omega\;ds_{g},
$$
where $\omega = \left\langle Y,\nu_{\Sigma} \right\rangle,$  $J_{\Sigma}$ is the Jacobi operator about $\Sigma$ given by
\begin{equation} \label{eq.3.2}
J_{\Sigma}\coloneqq - \Delta_{\Sigma} + |II_{\Sigma}|^2 + \Ric_{g}(\nu_{\Sigma},\nu_{\Sigma}),
\end{equation}
where $\Delta_{\Sigma}$ is the Laplace-Beltrami operator on $\Sigma,$ $|II_{\Sigma}|^2$ is the norm of the second fundamental form of $\Sigma$ and $\Ric_{g}$ is the Ricci curvature of $M;$ and
\begin{equation} \label{eq.3.4}
B_{\Sigma} \omega \coloneqq \partial_{\tau_{\Sigma}}\omega - \dfrac{1}{\sin(\theta)} \left( II_{\partial M} (\tau_{\partial M},\tau_{\partial M}) + \cos(\theta) II_{\Sigma}(\tau_{\Sigma},\tau_{\Sigma}) \right) \omega.
\end{equation}

Since for any smooth $\omega$ with $\int_{\Sigma} \omega\;da_{g} = 0,$ there exists  an admissible, volume-preserving variation vector field $\omega \nu_{\Sigma}$ as a normal part, then Jacobi fields are solutions of 
\begin{equation} \label{eq.jacozin}
\left\lbrace
\begin{array}{l}
J_{\Sigma} \omega = 0\;\;\text{in}\;\Sigma,\\\\
B_{\Sigma}\omega = 0\;\;\text{on}\;\partial \Sigma,\\\\
\int_{\Sigma} \omega \;da_{g} = 0.
\end{array}
\right.
\end{equation}

See the Appendix of \cite{ros-souam} for a derivation of the second variation of the energy $E.$ 

\subsubsection{The nondegeneracy condition}
\begin{dfn} \label{dfn.3.3}
A $\theta$-admissible constant mean curvature hypersurface $\Sigma$ is  volume-nondegenerate if there are no nontrivial solutions $(\omega,c) \in \mathcal{C}^{2,\alpha}(\Sigma) \times \mathbb{R}$ of \eqref{eq.jacozin}.
\end{dfn}

Volume-nondegeneracy is equivalent to the fact that the operator
\begin{equation} \label{eq.3.6}
\begin{split}
L_{\Sigma} \colon [\mathcal{C}^{2,\alpha}(\Sigma)]_{\theta} \times \mathbb{R} &\to \mathcal{C}^{0,\alpha}(\Sigma) \times \mathbb{R}\\
(\omega, c)&\mapsto \left(L_{\Sigma} \omega + c, \int_{\Sigma} \omega\;da_{g}  \right),
\end{split}
\end{equation}
is injective, where the subscript $\theta$ is meant to point out that functions in $[\mathcal{C}^{2,\alpha}(\Sigma)]_{\theta}$ satisfy $B_{\Sigma}\omega = 0$ on $\partial \Sigma.$ Observe that $L_{\Sigma}$ is self-adjoint with respect to the scalar product
$$
\left\langle (v,c),(\omega,d) \right\rangle  \coloneqq \int_{\Sigma} v\omega\;da_{g} + cd,
$$
in $L^2(\Sigma) \times \mathbb{R}.$ Since $L_{\Sigma}$ is elliptic, the notion of volume-nondegeneracy is equivalent to invertibility.

We remark that volume-nondegeneracy is not a condition too restrictive, once that it is stable under small perturbations of metric and it holds for a generic choice of metric, as a consequence of the work \cite{white} of White.

\subsubsection{Statement of the result}

\begin{thm} \label{thm.main}
Let $(M,g)$ be a compact smooth oriented Riemannian manifold of dimension $n \geq 2,$ with smooth boundary $\partial M \neq \emptyset.$ If $\Sigma_{0} \subset M$ is a $\theta$-admissible volume-nondegenerate constant mean curvature hypersurface, then for some $\varepsilon_{0} > 0,$ there exists a family $\{ u_{\varepsilon} \}_{\varepsilon \in (0,\varepsilon_{0})}$ of critical points of $E_{\varepsilon}$ under the constraint $V(u) = |M^{+}(\Sigma_{0})| - |M^{-}(\Sigma_{0})|$ converging uniformly to $\pm 1$ on compact subsets of $M^{\pm}(\Sigma_{0}),$ respectively. 
\end{thm}

The solutions constructed in Theorem \ref{thm.main} are solutions of

\begin{equation} \label{eq.AC}
\left\lbrace
\begin{array}{l}
-\varepsilon^2 \Delta_{g} u + W' (u ) = \varepsilon \lambda\;\;\text{in}\;M,\\\\
-\varepsilon \partial_{\nu_{\partial M}} u = \sigma'(u)\;\;\text{on}\;\partial M,\\\\
\int_{M} u\;dv_{g} = c_{0}|M|,
\end{array}
\right.
\end{equation}
where the constant $c_{0} \in (-1,1)$ is fixed so that 
$$
c_{0} |M| = |M^{+}(\Sigma_{0})| - |M^{-}(\Sigma_{0})|.
$$

The expansions of $u_{\varepsilon}$ in terms of $\varepsilon,$ of the Lagrange multiplier $\lambda,$ namely 
$$
\lambda = \dfrac{1}{2}c_{*}nH_{\Sigma_{0}} + \mathcal{O}(\varepsilon),
$$
where the constant $c_{*}$ is given by
$$
c_{*} \coloneqq \int_{-1}^{1} \sqrt{2 W(s)}\;ds,
$$
are the same as in \cite{pacard-ritore}. For the energy $E_{\varepsilon}(u_{\varepsilon}),$ the expansion is given by
$$
E_{\varepsilon}(u_{\varepsilon}) = 2 \varepsilon c_{*} |\Sigma_{0}| + 	\varepsilon c_{*} \cos \theta  |M^{+}(\partial \Sigma_{0})| + \mathcal{O}(\varepsilon^2).
$$

We remark that the choice of $\sigma$ in \eqref{eq.sigma} can be more general, once it satisfies $|\sigma'(s)| \leq C \sqrt{2 W(s)}$ for some $C \in [0,1)$ and for all $s \in \mathbb{R},$ which ensures that $\theta$ is strictly positive, and we avoid the so called perfect wetting effect (see \cite{cahn}), when the contact energy density $|\sigma(1) - \sigma(-1)|$ of the interface $\{u_{\varepsilon} = 1\}$ with $\partial \Omega$ approaches to the surface density $c_{0}$ of the interface inside $M.$

Let us give a rough outline of the proof in order to the organization of the paper be clearer. 

Essentially we follow the proof of Pacard in \cite{pacard}, merging with some techniques in \cite{pacard-ritore}, and the main tool is the infinite dimensional Lyapunov-Schmidt reduction argument.

We want a solution $u$ of \eqref{eq.AC} written as the sum of three functions
$$
u = u_{\varepsilon} + v^{\sharp} + v^{\flat},
$$
where $u_{\varepsilon}$ is an approximate solution, which is $\pm 1$ outside a tubular neighborhood of $\Sigma_{0}$ and inside it is the one-dimensional model solution of
$$
u_{1}'' - W'(u_{1}) = 0\;\;\text{in}\;\;\mathbb{R},\;\;\;u_{1}(0) = 0,\;\;\; u_{1}(x) \to \pm 1\;\;\text{as}\;\;x \to \pm \infty,
$$
analyzed in Section \ref{sec.approx}; also, $v^{\sharp}$ is supported in a small tubular neighborhood of $\Sigma_{0}$ and $v^{\flat}$ is globally defined in $M.$ Since we parametrize a neighborhood of $\Sigma_{0}$ using twisted Fermi coordinates, which we recall in Section \ref{sec.fermi}, which are modified Fermi coordinates so that it is determined by the flow of a vector field normal in $\Sigma_{0}$ and tangent to $\partial M,$ in order to have a fibration of $\Sigma_{0},$ we can write the neighborhood as a product of the type $I \times \Sigma_{0},$ where $I \subset \mathbb{R}$ is a small interval, and we can identify each parallel hypersurface $\Sigma$ as a graph of some small $\zeta.$ For this form of $u,$ in Section \ref{sec.equiv}, considering linearizations of $W'$ and $\sigma',$ we have that \eqref{eq.AC} is solved if we find $(v^{\sharp},v^{\flat},\zeta)$ satisfying 
\begin{equation} \label{eq.equiv00}
\left\lbrace
\begin{split}
\mathcal{L}_{\varepsilon} v^{\flat} &= f_{1,\varepsilon},\;\;\text{in}\;\;M,\\
L_{\varepsilon} v^{\sharp} &= f_{2,\varepsilon},\;\;\text{in}\;\;M,\\
- \varepsilon J_{\Sigma} \zeta\; \dot{u}_{\varepsilon} &= f_{3,\varepsilon},\;\;\text{in}\;\;M,\\
\mathcal{P}_{\varepsilon} v^{\flat} &= f_{4,\varepsilon},\;\;\text{on}\;\;\partial M,\\
P_{\varepsilon} v^{\sharp} &= f_{5,\varepsilon},\;\;\text{on}\;\;\partial M,\\
- \varepsilon B_{\Sigma} \zeta\; \dot{u}_{\varepsilon} &=  f_{6,\varepsilon},\;\;\text{on}\;\;\partial M,\\
\int_{M} \Tilde{u}_{\varepsilon}\;dv_{g} - c_{0}|M| &=  - \int_{M} \left(\chi_{4} v^{\sharp} + v^{\flat}\right)\;dv_{g},
\end{split}
\right.
\end{equation}
for some functions $f_{i,\varepsilon}$ depending on $\varepsilon,$ $v^{\sharp},$ $v^{\beta}$ and $\zeta,$ for each $i = 1,\cdots,6,$ where
$J_{\Sigma}$ and $B_{\Sigma}$ are defined in \eqref{eq.3.2} and \eqref{eq.3.4}, $L_{\varepsilon}$ and $P_{\varepsilon}$ are defined in \eqref{eq.linearops} and $\mathcal{L}_{\varepsilon}$ and $\mathcal{P}_{\varepsilon}$ are defined in \eqref{eq.nlinearops}. We first deal with the subproblems
$$
\left\lbrace
\begin{split}
L_{\varepsilon} v^{\sharp} &= f_{2,\varepsilon},\;\;\text{in}\;\;M,\\
P_{\varepsilon} v^{\sharp} &= f_{5,\varepsilon},\;\;\text{on}\;\;\partial M,
\end{split}
\right.
$$
in Section \ref{sec.lin}, then with
$$
\left\lbrace
\begin{split}
\mathcal{L}_{\varepsilon} v^{\flat} &= f_{1,\varepsilon},\;\;\text{in}\;\;M,\\
\mathcal{P}_{\varepsilon} v^{\flat} &= f_{4,\varepsilon},\;\;\text{on}\;\;\partial M,
\end{split}
\right.
$$
in Section \ref{sec.nlin}, and using the nondegeneracy hypothesis in Definition \ref{dfn.3.3}, we proceed to solve \eqref{eq.equiv00} as a fixed point problem in Section \ref{sec.proof}.

\begin{rmk}
The behavior of $\sigma$ and its relation with the potential controlling the diffusivity in the interior is essential to the proof. In general, it is not expected that our approach of linearization of the boundary condition to work in general elliptic equations with nonlinear Neumann conditions (see \cite{barles} for instance).
\end{rmk}

In particular, in an bounded domain $\Omega \subset \mathbb{R}^{n}$ with smooth boundary, we can conclude the following.

\begin{corol}
Given a nondegenerate minimal submanifold $\Gamma$ on $\partial \Omega$ of dimension $m,$ with $1\leq m \leq n - 2,$ there is a family of solutions of \eqref{eq.AC} whose nodal sets concentrate along this submanifold.
\end{corol}

In fact, by the work \cite{fall-mahmoudi} of Fall and Mahmoudi, there exists a sequence of $\theta$-admissible volume-nondegenerate constant mean curvature hypersurface $\Sigma_{k}$ concentrating along $\Gamma.$ By Theorem \ref{thm.main}, for each $\Sigma_{k},$ there exists a family $\{u_{\varepsilon}^{k}\}_{\varepsilon}$ of solutions of \eqref{eq.AC} whose nodal sets concentrate along $\Sigma_{k}.$ Up to a subsequence, we have that the nodal sets of the sequence $\{u_{\varepsilon}^{k}\}_{k}$ concentrate along $\Gamma.$

\subsection{Acknowledgements}
I would like to express my deep gratitude to Professor Fernando Cod\'a Marques for presenting me this problem, guidance throughout its development and support while visiting him at Princeton University during the 2017-2018 academic year, under the support of CAPES-Brazil. I am very thankful to the the Department of Mathematics for its hospitality and generosity.

\section{Twisted Fermi coordinates} \label{sec.fermi}
In this section, we recall the definition of twisted Fermi coordinates introduced in \cite{pacard-ritore}. Suppose that $\Sigma$ is a $\theta$-admissible hypersurface of $M$ and recall that we consider $M$ as a subdomain of the larger Riemannian manifold $(\Tilde{M},\Tilde{g}),$ with $\Tilde{g}|_{M} = g.$ Denote by $\Tilde{\Sigma}$ the extension of $\Sigma$ in $\Tilde{M}.$

We define Fermi coordinates by
$$
Z(z,y) \coloneqq \exp_{y}(z\nu_{\Sigma}(y)),\;\;z \in \mathbb{R},\;\;y \in \Sigma.
$$
The implicit function theorem implies that $Z$ is a local diffeomorphism from a neighborhood of a point of $(0,y) \in \mathbb{R} \times \Sigma$ onto a neighborhood of $\mathbb{R}^{n}.$ Provided that $\tau_{0}$ is chosen small enough, we can define Fermi coordinates in 
$$
V_{\tau_{0}}(\Tilde{\Sigma}) \coloneqq \{Z(z,y) \in \Tilde{M};\; z \in (-\tau_{0},\tau_{0}),\;\; y \in \Tilde{\Sigma}\}.
$$
Since $\Tilde{\partial_{z}} = \frac{\partial}{\partial z}$ is defined in $V_{\tau_{0}}(\Tilde{\Sigma})$ and $\Sigma$ meets $\partial M$ with constant angle $\theta > 0,$ we can reduce $\tau_{0}$ so that $\nu_{\partial M}$ does not coincide with $\Tilde{\partial_{z}}$ in $V_{\tau_{0}}(\Tilde{\Sigma}).$

Consider a collar neighborhood of $\partial M$ of radius $\varepsilon_{0} > 0$ and set $\mathcal{V}$ the extension of $\nu_{\partial M}$ given by unit normal vector fields to each parallel hypersurface to $\partial M.$ Consider a smooth cut-off function $\eta$ identically $1$ if $(-\infty,\varepsilon_{0}/2)$ and equal to $0$ in $(\varepsilon_{0},\infty),$ and let
$\Tilde{\chi} \coloneqq \eta(d_{\partial M}),$ where $d_{\partial M}$ is the distance function to $\partial M.$ Define the vector field
$$
X \coloneqq \dfrac{\Tilde{\partial_{z}} - \left\langle  \Tilde{\partial_{z}}, \Tilde{\chi} \mathcal{V} \right\rangle \Tilde{\chi} \mathcal{V}}{\sqrt{1 - (2 - \Tilde{\chi}^{2})\left\langle  \Tilde{\partial_{z}}, \Tilde{\chi} \mathcal{V} \right\rangle^{2}}}.
$$
This vector field is tangent to $\partial M,$ $X = \Tilde{\partial_{z}}$ whenever $d_{\partial M} \geq \varepsilon_{0}$ and 
$$
\left\langle X, \Tilde{\partial_{z}} \right\rangle = \dfrac{1 - \left\langle  \Tilde{\partial_{z}}, \Tilde{\chi} \mathcal{V} \right\rangle^{2}}{\sqrt{1 - (2 - \Tilde{\chi}^{2})\left\langle  \Tilde{\partial_{z}}, \Tilde{\chi} \mathcal{V} \right\rangle^{2}}},
$$ 
so that on $\partial \Sigma,$ we have $X = \sin \theta \nu_{\Sigma} - \cos \theta \tau_{\partial \Sigma}.$

Considering $\{F_{z}\}_{z}$ the flow associated to $X,$ we note that it satisfies $F_{z}(\Sigma) \subset M,$ $F_{z}(\partial \Sigma) \subset \partial M$ and in fact, 
$$
(z,y) \in (-\tau_{0},\tau_{0}) \times \Sigma \mapsto F_{z}(y) \in V_{\tau_{0}}(\Tilde{\Sigma}) \cap M
$$
is a diffeomorphism. The twisted Fermi coordinates relative to $\Sigma$ are given by
$$
Y(z,y) \coloneqq F_{z}(y).
$$
It holds that wherever $d_{\partial M} \geq \varepsilon_{0},$ we have $Y(z,y) = Z(z,y).$


In summary, $Y \colon (-\tau_{0},\tau_{0}) \times \Sigma \to M$ gives a variation of $\Sigma$ such that, for every $z \in (-\tau_{0},\tau_{0}),$ the map $F_{z} = Y(z,\cdot) \colon \Sigma \to M$ is an immersion of $\Sigma$ in $M$ such that $F_{z}(\partial \Sigma) \subset \partial M.$ We denote $\Sigma_{z} \coloneqq F_{z}(\Sigma)$ and the subscript $z$ will be used to denote quantities associated to it, such as the induced metric $g_{z}$ and the unit normal $\nu_{z}$ to $\Sigma_{z}.$ 


\section{The solution profile} \label{sec.approx}
In this section, we use the twisted Fermi coordinates construct an approximate solution of \eqref{eq.AC} whose nodal set is close to $\Sigma.$ 

Recall that the unique heteroclinic solution $u_{1}$ of
$$
u_{1}'' - W'(u_{1}) = 0\;\;\text{in}\;\;\mathbb{R},\;\;\;u_{1}(0) = 0,\;\;\; u_{1}(x) \to \pm 1\;\;\text{as}\;\;x \to \pm \infty,
$$
has the asymptotic properties
$$
\begin{array}{l}
|\partial_{t}^{k}(u_{1}(t) - 1)| \leq c_{k} e^{-\gamma_{+}|t|},\;\;t \geq 0,\\\\
|\partial_{t}^{k}(u_{1}(t) + 1)| \leq c_{k} e^{\gamma_{-}|t|},\;\;t \leq 0,
\end{array}
$$
for some constant $c_{k} > 0$ for each $k \in \mathbb{N},$ where $\gamma_{\pm} > 0$ are the indicial roots defined by
$$
\gamma_{\pm}^{2} \coloneqq W''(\pm 1).
$$ 

Define
$$
u_{\varepsilon} \coloneqq u_{1}\left(\dfrac{z}{\varepsilon} \right),\;\;\;\dot{u}_{\varepsilon} \coloneqq u'_{1}\left(\dfrac{z}{\varepsilon} \right) = \varepsilon u'_{\varepsilon}(z),\;\;\;\ddot{u}_{\varepsilon} \coloneqq u''_{1}\left(\dfrac{z}{\varepsilon} \right) = \varepsilon^{2} u''_{\varepsilon}(z).
$$

Given any smooth function $\zeta$ small enough defined on $\overline{\Sigma},$ consider
$$
\Sigma_{\zeta} \coloneqq \{Y(\zeta(y),y) \in M;\; y \in \overline{\Sigma}\}.
$$
In a tubular neighborhood of $\Sigma,$ we write
$$
Y^{*} u(z,y) = \overline{u}(z - \zeta(y),y).
$$
For convenience, denote $t \coloneqq z - \zeta(y).$

Using the expressions of the Laplacian and the normal derivative with respect to $\nu_{\partial M}$ on twisted Fermi coordinates, we can rewrite \eqref{eq.AC} as
\begin{equation} \label{eq.aux}
\left\lbrace
\begin{array}{l}
-\varepsilon^{2} \left[ (1 + \Vert d\zeta \Vert^{2}_{g_{z}}) \partial_{t}^{2} \overline{u} + \Delta_{g_{z}} \overline{u} - (\Delta_{g_{z}} \zeta + H_{z} ) \partial_{t} \overline{u} - 2(d\zeta, d\partial_{t} \overline{u})_{g_{z}} \right] |_{z = t + \zeta} + W'(\overline{u}) = \varepsilon \lambda\;\;\text{in}\;M,\\\\
-\varepsilon \left[ -(\partial_{t} \overline{u} \;d\zeta ,\nu_{\partial M})_{g_{z}} + (\nabla \overline{u} ,\nu_{\partial M})_{g_{z}} +  \right]|_{z = t + \zeta} - \sigma'(\overline{u}) = 0 \;\;\text{on}\;\partial M,
\end{array}
\right.
\end{equation}
for $t > 0$ close to $0$ and $y \in \Sigma.$ 

Setting
$$
\overline{u}(t,y) \coloneqq u_{\varepsilon}(t) + v(t,y),
$$
then \eqref{eq.aux} becomes
\begin{equation}
\left\lbrace
\begin{array}{l}
\mathfrak{N}(v,\zeta) = \varepsilon \lambda\;\;\text{in}\;M,\\\\
\mathfrak{M}(v,\zeta) = 0\;\;\text{on}\;\partial M,
\end{array}
\right.
\end{equation}
where
$$
\begin{array}{ll}
\mathfrak{N}(v,\zeta) \coloneqq R(v,\zeta) + \varepsilon H_{z} \dot{u}_{\varepsilon} + Q_{\varepsilon}(v),  
\end{array}
$$
and
$$
\begin{array}{ll}
\mathfrak{M}(v,\zeta) \coloneqq  S(v,\zeta) + \left(\left\langle \nu_{\Sigma_{t}}, \nu_{\partial M} \right\rangle - \cos \theta \right)\dot{u}_{\varepsilon} + \Tilde{Q}_{\varepsilon}(v)
\end{array}
$$
with 
$$
Q_{\varepsilon}(v) \coloneqq  W'(u_{\varepsilon} + v) - W'(u_{\varepsilon}) - W''(u_{\varepsilon})v ,
$$
$$
\Tilde{Q}_{\varepsilon}(v) \coloneqq - \left( \sigma'(u_{\varepsilon} + v) - \sigma'(u_{\varepsilon}) - \sigma''(u_{\varepsilon})v \right)		
$$
and $R,$ $S$ functions so that $R(0,0) = S(0,0) = 0.$ We then have
$$
\mathfrak{N}(0,0) = \varepsilon H_{t} \dot{u}_{\varepsilon}\;\;\text{and}\;\;\mathfrak{M}(0,0) = \left(\left\langle \nu_{\Sigma_{t}}, \nu_{\partial M} \right\rangle - \cos \theta \right)\dot{u}_{\varepsilon}.
$$
Also recall that
$$
H_{t} = H_{0} + t\left( |II_{\Sigma}|^2 + \Ric_{g}(\nu_{\Sigma},\nu_{\Sigma}) \right) + \mathcal{O}(t^2),
$$
and
$$
\left\langle \nu_{\Sigma_{t}}, \nu_{\partial M} \right\rangle - \cos \theta  = t \left( -\sin \theta \frac{\partial v}{\partial \tau_{\partial \Sigma}} + \left[ \cos \theta II_{\Sigma}(\tau_{\partial \Sigma},\tau_{\partial \Sigma}) + II_{\partial M}(\tau_{\partial M},\tau_{\partial M}) \right]v\right) + \mathcal{O}(t^{2}),
$$
where  $v = g(X,\nu_{\Sigma})$ and $H_{0}$ is constant, since $\Sigma$ is a capillar hypersurface with constant mean curvature $H_{0}$ and fixed contact angle $\theta.$ The following estimates are direct.

\begin{lem} \label{lem.4}
In a fixed neighborhood of $\Sigma,$ there exists a constant $C > 0$ such that
$$
|\mathfrak{N}(0,0)|_{\mathring{g}} \leq \varepsilon,\;\;\text{and}\;\;|\mathfrak{M}(0,0)|_{\mathring{g}} \leq C  \varepsilon.
$$

\end{lem}

Given a function $f$ defined in $\mathbb{R} \times \Sigma,$ we define
\begin{equation} \label{eq.ortinha}
\Pi(f) \coloneqq \dfrac{1}{\varepsilon c} \int_{-\infty}^{\infty} f(t,y) \dot{u}_{\varepsilon}(t) \;dt,\;\;\;y \in \Sigma
\end{equation}
where
$$
c \coloneqq \dfrac{1}{\varepsilon} \int_{-\infty}^{\infty} \dot{u}_{\varepsilon}^{2}(t)\;dt = \int_{-\infty}^{\infty} (u'_{1})^{2}(t)\;dt,
$$
so that $\Pi(\dot{u}_{\varepsilon}) = 1.$ We have
$$
\int_{\mathbb{R}} H_{t} \dot{u}_{\varepsilon}^{2}\;dt = \mathcal{O}(\varepsilon)\;\;\;\text{and}\;\;\int_{\mathbb{R}}  (\left\langle \nu_{\Sigma_{t}}, \nu_{\partial M} \right\rangle - \cos \theta )\dot{u}_{\varepsilon}^{2}\;dt = \mathcal{O}(\varepsilon^{2})
$$
since the integral of $t \dot{u}_{\varepsilon}^{2}$ is $0.$ Using this property, we have the following.

\begin{lem} \label{lem.5}
There exists a constant  $C > 0$ such that
$$
|\Pi\left(\chi\mathfrak{N}(0,0)\right)|_{\mathring{g}} \leq C \varepsilon,\;\;\text{and}\;\;|\Pi\left(\chi\mathfrak{M}(0,0)\right)|_{\mathring{g}} \leq C  \varepsilon^{2}.
$$
\end{lem}

\section{The subproblem on $\mathbb{R} \times \overline{\Sigma}$} \label{sec.lin}
Consider the operators defined by
\begin{equation} \label{eq.linearops}
L_{\varepsilon} \coloneqq -\varepsilon^{2}\left( \partial_{t}^{2}  + \Delta_{\mathring{g}} \right) + W''(u_{\varepsilon})\;\;\text{and}\;\;P_{\varepsilon} \coloneqq  - \varepsilon \left\langle \nabla, \nu_{\mathbb{R} \times \partial \Sigma}\right\rangle - \sigma''(u_{\varepsilon}),
\end{equation} 
acting on functions defined on $\mathbb{R} \times \Sigma$ and $\mathbb{R} \times \partial \Sigma,$ respectively, where $\mathbb{R} \times \overline{\Sigma}$ is endowed with the product metric $ dt^{2} + \mathring{g}.$

\subsection{Injectivity}
First, let us recall the operator
$$
L_{0} \coloneqq - \left( \partial_{t}^{2} - W''(u_{1}) \right)
$$
acting on functions defined on $\mathbb{R}.$ Note that an eigenpair $(-\mu, \omega)$ of $L_{0}$ satisfy $L_{\mu} \omega = 0,$ where
$$
L_{\zeta} \coloneqq - \partial_{t}^{2} + \zeta + W''(u_{1}) ,\;\;\;\;\zeta \in \mathbb{R},
$$ 
is the operator considered in Section $7.1$ of \cite{pacard-ritore}. 



By the spectral theory for this ordinary differential operator, since the set of solutions of $L_{\zeta} \omega  = 0$ is two dimensional for each $\zeta$ and $(0,u_{1}')$ is an eigenpair of $L_{0},$ the positivity of the first eigenvalue $\mu_{1}$ implies that
\begin{equation} \label{spec}
\int_{\mathbb{R}} |\partial_{t} \omega|^{2} + W''(u_{1}) \omega^{2}\;dt \geq \mu_{1} \int_{\mathbb{R}} \omega^{2}\;dt,
\end{equation}
for all $\omega \in H^{1}(\mathbb{R})$ satisfying
\begin{equation} \label{eq.orthogt}
\int_{\mathbb{R}} \omega(t)\;u_{1}'(t)\;dt = 0.
\end{equation}

Now consider the operators
$$
L_{*} \coloneqq -\left( \partial_{t}^{2} + \Delta_{\mathbb{R}^{n-1}} - W''(u_{1}) \right)\;\;\text{and}\;\;P_{*} \coloneqq \left\langle \nabla, \nu_{\mathbb{R}\times \partial \mathbb{R}^{n-1}_{+}} \right\rangle - \sigma''(u_{1}),
$$
acting on functions defined on $\mathbb{R} \times \mathbb{R}^{n-1}$ and on $\mathbb{R} \times \partial\mathbb{R}^{n-1}_{+},$ respectively.

\begin{lem}[Corollary $7.5$ of \cite{pacard-ritore}] \label{lem.7}
Assume that $\omega \in L^{\infty}(\mathbb{R}\times \mathbb{R}^{n-1})$ satisfies $L_{*} \omega = 0\;\;\text{in}\;\;\mathbb{R} \times \mathbb{R}^{n-1}_{+}.$ Then $\omega$ only depends on $t$ and it is collinear to $u_{1}'.$
\end{lem}

\subsection{An \textit{a priori} estimate}
Consider on $\mathbb{R} \times \Sigma$ the scaled metric
$$
g_{\varepsilon} \coloneqq \varepsilon^{2}\left( dt^{2} + \mathring{g}  \right).
$$

\begin{dfn}
For all $k \in \mathbb{N},$ $\alpha \in (0,1),$ the space $\mathcal{C}^{k,\alpha}_{\varepsilon}(\mathbb{R}\times \Sigma)$ ($\mathcal{C}^{k,\alpha}_{\varepsilon}(\mathbb{R} \times \partial \Sigma)$) is the space of functions $\omega \in \mathcal{C}^{k,\alpha}(\mathbb{R}\times \Sigma)$ $(\mathcal{C}^{k,\alpha}(\mathbb{R}\times \partial \Sigma))$ where the H\"older norm is computed with respect to $g_{\varepsilon},$ that is
$$
\Vert u \Vert_{\mathcal{C}^{k,\alpha}_{\varepsilon}} = \sum_{j = 0}^{k} \varepsilon^{j} \Vert \nabla^{j} u \Vert_{L^{\infty}} + \varepsilon^{k + \alpha} \sup_{x\neq y} \dfrac{|\nabla^{k}u(x) - \nabla^{k}u(y)|}{\dist(x,y)^{\alpha}}. 
$$
\end{dfn}
Thus, if $\omega \in \mathcal{C}^{k,\alpha}_{\varepsilon}(\mathbb{R}\times \Sigma),$ then
$$
|\nabla^{a} \partial_{t}^{b} \omega(t,y)|_{\mathring{g}} \leq C \leq \Vert \omega \Vert_{\mathcal{C}^{k,\alpha}_{\varepsilon}(\mathbb{R}\times \Sigma)} \varepsilon^{-(a+b)}.
$$
provided $a + b \leq k.$

We shall work in the closed subspace of functions $\omega \coloneqq \mathbb{R} \times \overline{\Sigma} \to \mathbb{R}$ satisfying
\begin{equation} \label{eq.orthogonal}
\int_{-\infty}^{\infty} \omega(t,y) \dot{u}_{\varepsilon}(t)\;dt = 0 \;\;\;\;\forall y \in \overline{\Sigma}.
\end{equation}

\begin{prop} \label{prop.1}
There exist constants $C > 0$ and $\varepsilon_{0} > 0$ such that, for all $\varepsilon \in (0,\varepsilon_{0})$ and for all $\omega \in \mathcal{C}^{2,\alpha}_{\varepsilon}(\mathbb{R} \times \Sigma) \cap \mathcal{C}^{1,\alpha}_{\varepsilon}(\mathbb{R} \times \overline{\Sigma})$ satisfying \eqref{eq.orthogonal}, we have 
$$
\Vert \omega \Vert_{\mathcal{C}^{2,\alpha}_{\varepsilon}(\mathbb{R} \times \Sigma) \cap \mathcal{C}^{1,\alpha}_{\varepsilon}(\mathbb{R} \times \partial \Sigma)} \leq C \left( \Vert L_{\varepsilon} \omega \Vert_{\mathcal{C}^{0,\alpha}_{\varepsilon}(\mathbb{R} \times \Sigma)} +  \Vert P_{\varepsilon} \omega \Vert_{\mathcal{C}^{0,\alpha}_{\varepsilon}(\mathbb{R} \times \partial \Sigma)} \right).
$$
\end{prop}
\begin{proof}
By elliptic regularity theory, it is enough to prove that
$$
\Vert \omega \Vert_{L^{\infty}(\mathbb{R} \times \overline{\Sigma})} \leq C \left( \Vert L_{\varepsilon} \omega \Vert_{L^{\infty}(\mathbb{R} \times \Sigma)} +  \Vert P_{\varepsilon} \omega \Vert_{L^{\infty}(\mathbb{R} \times \partial \Sigma)} \right).
$$
The proof is by contradiction. Assume that, for each element of a sequence $\varepsilon_{i}$ tending to $0,$ there exists a function $\omega_{i}$ satisfying \eqref{eq.orthogonal},
$$
\Vert \omega_{i} \Vert_{L^{\infty}(\mathbb{R} \times \overline{\Sigma})} = 1,\;\;\lim_{i \to \infty} \Vert L_{\varepsilon} \omega_{i} \Vert_{L^{\infty}(\mathbb{R} \times \Sigma)} = 0,
$$
and $P_{\varepsilon} \omega_{i} = 0$ on $\mathbb{R} \times \partial \Sigma.$
For each $i \in \mathbb{R},$ we a choose point $(t_{i},y_{i}) \in \mathbb{R} \times \overline{\Sigma}$ with 
$$
|\omega_{i}(t_{i},y_{i}) | \geq \dfrac{1}{2}.
$$
One can show that $|t_{i}| \leq C \varepsilon_{i}$ and so this sequence tend to $0.$
Assume that there exists a constant $\Lambda > 0$ such that for all $y \in \Sigma,$ there is a local chart $Y_{y} \colon B^{n-1}(\Lambda) \to \Sigma$ where $B^{n-1}(\Lambda)$ denotes the ball of radius $\Lambda$ centered at the origin in $\mathbb{R}^{n-1}.$ For $y \in \partial \Sigma,$ suppose $Y_{y} \colon B_{+}^{n-1}(\Lambda) \to \Sigma,$ where $B_{+}^{n-1}(\Lambda)$ is the half ball of radius $\Lambda$ centered at the origin of $\mathbb{R}^{n-1}_{+}.$ 
For each $i,$ consider the function $\Tilde{\omega}_{i}$ defined by
$$
\Tilde{\omega}_{i}(t,y) \coloneqq \omega_{i} (\varepsilon_{i}t, Y_{y_{i}}(\varepsilon_{i} y)),
$$
for all $t \in \mathbb{R},$ all $y \in \mathbb{R}^{n-1}$ if $y_{i} \in \Sigma$ and all $y \in \mathbb{R}_{+}^{n-1}$ if $y_{i} \in \partial \Sigma.$ Also, consider
$$
\rho_{i} = \dfrac{d_{g}(y_{i},\partial \Sigma)}{\varepsilon_{i}},
$$
where $d_{g}$ is the geodesic distance on $\Sigma,$ and we can assume that $\rho_{i}$ converges to $\rho_{\infty} \in [0,\infty].$

If $\rho_{\infty} = \infty,$ using elliptic estimates together with Ascoli's Theorem, we can extract a subsequence and pass to the limit in the equation satisfied by $\Tilde{\omega}_{i}.$ We find that $\Tilde{\omega}_{i}$ converges, uniformly on compacts of $\Sigma$ to $\Tilde{\omega}$ which is a non trivial solution of $L_{*} \Tilde{\omega} = 0$ in $\mathbb{R} \times \mathbb{R}^{n-1}.$ If If $\rho_{\infty} < \infty,$ the same argument implies that $\Tilde{\omega}$ is a weak solution of $L_{*} \Tilde{\omega} = 0$ in $\mathbb{R} \times \mathbb{R}^{n-1}_{+}$ and it satisfies $P_{*} \Tilde{\omega} = 0$ in $\mathbb{R} \times \partial \mathbb{R}^{n-1}_{+}.$ We can extend it to $\mathbb{R} \times \mathbb{R}^{n-1}$ by reflection across the hyperplane $y_{n-1} = 0.$ By Lebesgue's convergence theorem, we can pass to the limit in \eqref{eq.orthogonal}, to get that $\Tilde{\omega}$ satisfies \eqref{eq.orthogt} and $\Tilde{\omega} \in L^{\infty}(\mathbb{R} \times \mathbb{R}^{n-1}).$ This contradicts the result of Lemma \ref{lem.7} and the proof is complete.
\end{proof}

\subsection{Surjectivity}

\begin{prop} \label{prop.2}
There exists $\varepsilon_{0} > 0$ such that, for all $\varepsilon \in (0,\varepsilon_{0})$ and for all $f \in \mathcal{C}^{0,\alpha}_{\varepsilon}(\mathbb{R} \times \Sigma)$ and $h \in \mathcal{C}^{0,\alpha}_{\varepsilon}(\mathbb{R} \times \partial \Sigma)$ satisfying \eqref{eq.orthogonal}, there exists a unique $\omega \in \mathcal{C}^{2,\alpha}_{\varepsilon}(\mathbb{R} \times \Sigma) \cap \mathcal{C}^{1,\alpha}_{\varepsilon}(\mathbb{R} \times \partial \Sigma)$ also satisfying \eqref{eq.orthogonal} and which is a solution of
$$
\left\lbrace
\begin{array}{l}
L_{\varepsilon} \omega  = f,\;\;\text{in}\;\;\mathbb{R} \times \Sigma,\\
P_{\varepsilon} \omega = h,\;\;\text{on}\;\;\mathbb{R} \times \partial \Sigma.
\end{array}
\right.
$$
\end{prop}
\begin{proof}
Note that we have the weak formulation
\begin{equation*}
\begin{split}
&\int_{\mathbb{R} \times \Sigma} \varepsilon^{2} \left(\partial_{t}\omega \;\partial_{t}\phi + (\nabla \omega,\nabla \phi)_{\mathring{g}} \right) + W''(u_{\varepsilon}) \omega \phi\;dt\;dv_{\mathring{g}} + \int_{\mathbb{R} \times \partial \Sigma} \sigma''(u_{\varepsilon}) \omega \phi\;dt\;da_{\mathring{g}} \\
&\quad \quad= \int_{\mathbb{R} \times \Sigma} f\phi\;dt\;dv_{\mathring{g}}-\int_{\mathbb{R} \times \partial \Sigma} h \phi\;dt\;da_{\mathring{g}},\;\;\;\;\forall \phi \in H^{1}(\mathbb{R} \times \Sigma).
\end{split}
\end{equation*}
Consider
$$
F(\omega) \coloneqq \int_{\mathbb{R} \times \Sigma} \varepsilon^{2} \left(|\partial_{t} \omega|^{2} + |\nabla \omega|^{2}_{\mathring{g}} \right) + W''(u_{\varepsilon}) \omega^{2}\;dt\;dv_{\mathring{g}}  + \int_{\mathbb{R} \times \partial \Sigma} \sigma''(u_{\varepsilon}) \omega^{2}\;dt\;da_{\mathring{g}},
$$
acting on the space of functions $H^{1}(\mathbb{R} \times \Sigma)$ which satisfy \eqref{eq.orthogonal} for a.e. $y \in \Sigma.$

By \eqref{spec}, we have
$$
F(\omega) \geq c_{1} \int_{\mathbb{R}\times \Sigma}\omega^{2} + c_{2} \int_{\mathbb{R}\times \partial \Sigma}\omega^{2}.
$$
Thus, the existence and uniqueness of a weak solution of $L_{\varepsilon} \omega = f$ in $H^{1}(\mathbb{R} \times \Sigma)$ satisfying $P_{\varepsilon} \omega = h $ follow by the Lax-Milgram's Theorem and by elliptic regularity, the proof is concluded.
\end{proof}

\section{The subproblem on $\overline{M}$} \label{sec.nlin}
Consider
\begin{equation} \label{eq.nlinearops}
\mathcal{L}_{\varepsilon} \coloneqq -\varepsilon^{2} \Delta_{g} - \Gamma,\;\; \mathcal{P}_{\varepsilon} \coloneqq -\varepsilon \left\langle \nabla , \nu_{\partial M} \right\rangle + \Tilde{\Gamma},
\end{equation}
where $\Gamma$ is an interpolation between $W''(-1)$ in $M^{-}(\Sigma)$ and $W''(1)$ in  $M^{+}(\Sigma)$ and analogously on $\partial M,$ $\Tilde{\Gamma}$ is an interpolation between $\sigma''(-1)$ and $\sigma''(1).$ Precisely, let $\xi$ be a nonnegative smooth cutoff function satisfying
$$
\xi(t) = 
\left\lbrace
\begin{array}{l}
1,\quad \text{in}\;(1,\infty),\\
0,\quad \text{in}\;(-\infty,-1),
\end{array}
\right.
$$
and in twisted Fermi coordinates with respect to $\Sigma,$ define respectively
$$
\Gamma(t,y) \coloneqq \left(1 - \xi\left(\frac{t}{\varepsilon}\right) \right) W''(-1) + \xi\left(\frac{t}{\varepsilon}\right) W''(1),
$$
$$
\Tilde{\Gamma}(t,y) \coloneqq \left(1 - \xi\left(\frac{t}{\varepsilon}\right) \right) \sigma''(-1) + \xi\left(\frac{t}{\varepsilon}\right) \sigma''(1),
$$
in $V_{\tau_{0}}(\Sigma)$ and
$$
\Gamma \coloneqq W''(\pm1)\;\;\text{and}\;\;\Tilde{\Gamma} \coloneqq \sigma''(\pm1)\;\;\text{in}\; M^{\pm}(\Sigma) \setminus V_{\tau_{0}}(\Sigma),
$$
Note that here we are and will be abusing the notation $\sigma''(\pm 1) \coloneqq \lim_{u \to \pm1^{\mp}} \frac{W'(u)}{\sqrt{W(u)}}.$

\begin{prop} \label{prop.3}
There exists a constant $C > 0$ such that
$$
\Vert \omega \Vert_{\mathcal{C}^{2,\alpha}_{\varepsilon}(M) \cap \mathcal{C}^{1,\alpha}_{\varepsilon}(\overline{M})} \leq C \left( \Vert \mathcal{L}_{\varepsilon} \omega \Vert_{\mathcal{C}^{0,\alpha}_{\varepsilon}(M)} +  \Vert \mathcal{P}_{\varepsilon} \omega \Vert_{\mathcal{C}^{0,\alpha}_{\varepsilon}(\partial M)} \right),
$$
where $\mathcal{C}^{k,\alpha}_{\varepsilon}(M)$ ($\mathcal{C}^{k,\alpha}_{\varepsilon}(\partial M)$) is the H\"older space of functions defined on $M$ ($\partial M$) where the norms of the derivatives and H\"older  derivatives are performed using the scaled metric $\varepsilon^{2} g.$ 
\end{prop}
\begin{proof}
This follows from Schauder estimate for solutions $u \in \mathcal{C}^{2,\alpha}_{\varepsilon}(M) \cap \mathcal{C}^{1,\alpha}_{\varepsilon}(\overline{M})$ of the problem
$$
\left\lbrace
\begin{array}{l}
\mathcal{L}_{\varepsilon} \omega  = f,\;\;\text{in}\;\;M,\\
\mathcal{P}_{\varepsilon} \omega = h,\;\;\text{on}\;\;\partial M,
\end{array}
\right.
$$
for any $f \in C^{0,\alpha}_{\varepsilon}(M)$ and $h \in  C^{0,\alpha}_{\varepsilon}(\partial M),$ see \cite{gilbarg-trudinger} (Theorems $6.30$ and $6.31$) and \cite{lieberman}.
\end{proof}

\section{An equivalent formulation} \label{sec.equiv}

In order to define an approximate solution and perturb it appropriately so we can achieve a full solution to \eqref{eq.AC}, we consider the strategy in \cite{pacard}.  
\begin{enumerate}
\item For $j = 1,\cdots,5,$ define $\chi_{j}$ to be even non-decreasing cut-off functions satisfying
$$
Y^{*} \chi_{j}(t,y) \coloneqq
\left\lbrace
\begin{array}{lll}
1&\text{when}&|t| \leq \varepsilon^{\delta_{*}} \left( 1 - \frac{2j - 1}{100} \right),\\
0&\text{when}&|t| \geq \varepsilon^{\delta_{*}} \left( 1 - \frac{2j - 2}{100} \right),
\end{array}
\right.
$$ 
for some fixed $\delta_{*} \in (0,1).$  We denote by $\Omega_{j}$ the support of $\chi_{j}.$ Note that $\Omega_{j + 1} \subset \Omega_{j}$ and that $Y$ is a diffeomorphism from the set $\Sigma \times (-2\varepsilon^{\delta_{*}},2\varepsilon^{\delta_{*}})$ onto its image.
\item Define an approximate solution by
$$
\Tilde{u}_{\varepsilon} \coloneqq \chi_{1} \overline{u}_{\varepsilon} \pm (1 - \chi_{1}),
$$
where $\pm$ corresponds to whether the point belongs to $M^{\pm}(\Sigma).$ Here
$$
Y^{*} \overline{u}_{\varepsilon}(t,y) \coloneqq u_{\varepsilon}(t).
$$
Since $\overline{u}_{\varepsilon}$ is exponentially close to $\pm 1$ at infinity, it is reasonable to consider it as $\pm 1 $ away from $\Gamma.$
\item Given a function $\zeta \in \mathcal{C}^{2,\alpha}(\Sigma) \cap \mathcal{C}^{1,\alpha}(\overline{\Sigma}),$ define a diffeomorphism $D_{\zeta}$ of $M$ by
$$
\left\lbrace
\begin{array}{l}
Y^{*} D_{\zeta}(t,y) \coloneqq Y(t - \chi_{2}(t,y)\zeta(y),y)\;\;\text{in}\;\;\Omega_{2},\\
D_{\zeta} = I\;\;\text{in}\;\;M \setminus \Omega_{1}.
\end{array}
\right.
$$
The inverse of $D_{\zeta}$ can be written as
$$
Y^{*} D^{-1}_{\zeta}(t,y) = Y(t + \chi_{2}(t,y) \zeta(y) + \xi(t,y,\zeta(y)) \zeta(y)^{2},y)\;\;\text{in}\;\;\Omega_{2},
$$
where $(t,y,z) \mapsto \xi(t,y,z)$ is a smooth function defined for $z$ small.
\end{enumerate}

Now, given a function $\zeta \in \mathcal{C}^{2,\alpha}(\Sigma) \cap \mathcal{C}^{1,\alpha}(\overline{\Sigma})$ small enough, we use the diffeomorphism $D_{\zeta}$ and we consider $u = \overline{u} \circ D_{\zeta}$ so that the equation \eqref{eq.AC} can be rewritten as
$$
\left\lbrace
\begin{array}{l}
-\varepsilon^{2} \Delta_{g}(\overline{u} \circ D_{\zeta}) \circ D_{\zeta}^{-1} + W'(\overline{u}) = \varepsilon \lambda,\;\;\text{in}\;\;M,\\\\
-\varepsilon \left\langle \nabla \left(\overline{u} \circ D_{\zeta}\right) \circ D_{\zeta}^{-1}, \nu_{\partial M} \right\rangle = \sigma'(\overline{u}) \;\;\text{on}\;\partial M,\\\\
\int_{M} \overline{u}\;dv_{g} = c_{0}|M|,
\end{array}
\right.
$$ 

We look for a solution 
$$
\overline{u} \coloneqq \Tilde{u}_{\varepsilon} + v,
$$
so that \eqref{eq.AC} can be written as
$$
\left\lbrace
\begin{array}{l}
-\varepsilon^{2} \Delta_{g} (v \circ D_{\zeta}) \circ D_{\zeta}^{-1} + W''(\Tilde{u}_{\varepsilon}) v + E_{\varepsilon}(\zeta) +Q_{\varepsilon}(v) = \varepsilon \lambda,\;\;\text{in}\;\;M,\\\\
-\varepsilon \left\langle \nabla \left(v \circ D_{\zeta}\right) \circ D_{\zeta}^{-1}, \nu_{\partial M} \right\rangle - \sigma''(\Tilde{u}_{\varepsilon})  v - \Tilde{E}_{\varepsilon}(\zeta) - \Tilde{Q}_{\varepsilon}(v) = 0, \;\;\text{on}\;\partial M,\\\\
\int_{M} \Tilde{u}_{\varepsilon}\;dv_{g} = c_{0}|M| - \int_{M} \left(\chi_{4} v^{\sharp} + v^{\flat} \right)\;dv_{g},
\end{array}
\right.
$$
where
$$
E_{\varepsilon}(\zeta) \coloneqq -\varepsilon^{2} \Delta_{g} (\Tilde{u}_{\varepsilon} \circ D_{\zeta}) \circ D_{\zeta}^{-1} + W'(\Tilde{u}_{\varepsilon}),\;\;\;\Tilde{E}_{\varepsilon}(\zeta) \coloneqq \varepsilon \left\langle \nabla (\Tilde{u}_{\varepsilon} \circ D_{\zeta}) \circ D_{\zeta}^{-1}, \nu_{\partial M} \right\rangle + \sigma'(\Tilde{u}_{\varepsilon}),
$$
and
$$
Q_{\varepsilon}(v) \coloneqq  W'(\Tilde{u}_{\varepsilon} + v) - W'(\Tilde{u}_{\varepsilon} ) - W''(\Tilde{u}_{\varepsilon})v ,\;\;\Tilde{Q}_{\varepsilon}(v) \coloneqq \sigma'(\Tilde{u}_{\varepsilon} + v) - \sigma'(\Tilde{u}_{\varepsilon} ) - \sigma''(\Tilde{u}_{\varepsilon})v .
$$

We set
$$
v \coloneqq \chi_{4} v^{\sharp} + v^{\flat},
$$
with the function $v^{\flat}$ satisfying
$$
\left\lbrace
\begin{array}{l}
\mathcal{L}_{\varepsilon} v^{\flat} = N_{\varepsilon}(v^{\flat},v^{\sharp},\zeta),\\
\mathcal{P}_{\varepsilon} v^{\flat} = \Tilde{N}_{\varepsilon}(v^{\flat},v^{\sharp},\zeta),
\end{array}
\right.
$$
where
\begin{equation*}
\begin{split}
N_{\varepsilon}(v^{\flat},v^{\sharp},\zeta) \coloneqq (\chi_{4} - 1) &\left[ \varepsilon^{2} \left(\Delta_{g} (v^{\flat} \circ D_{\zeta}) \circ D_{\zeta}^{-1} - \Delta_{g} v^{\flat} \right) - \left(W''(\Tilde{u}_{\varepsilon}) -  \Gamma \right)v^{\flat} - E_{\varepsilon}(\zeta) + \varepsilon \lambda\right.\\
& \;\;\;\left. - Q_{\varepsilon}(\chi_{4} v^{\sharp} + v^{\flat}) \right] - \varepsilon^{2} \left( \Delta_{g} (\chi_{4} v^{\sharp} \circ D_{\zeta}) - \chi_{4} \Delta_{g} ( v^{\sharp} \circ D_{\zeta}) \right) \circ D_{\zeta}^{-1},
\end{split}
\end{equation*}
and
\begin{equation*}
\begin{split}
\Tilde{N}_{\varepsilon}(v^{\flat},v^{\sharp},\zeta) \coloneqq (\chi_{4} - 1) &\left[ -\varepsilon \left( \left\langle \nabla (v^{\flat} \circ D_{\zeta}) \circ D_{\zeta}^{-1}, \nu_{\partial M} \right\rangle -  \left\langle \nabla v^{\flat}, \nu_{\partial M} \right\rangle \right) - \left( \sigma''(\Tilde{u}_{\varepsilon}) -  \Tilde{\Gamma} \right) v^{\flat} - \Tilde{E}_{\varepsilon}(\zeta) \right.\\
& \;\;\;\left. - \Tilde{Q}_{\varepsilon}(\chi_{4} v^{\sharp} + v^{\flat}) \right] - \varepsilon \left\langle \left( \nabla (\chi_{4} v^{\sharp} \circ D_{\zeta}) - \chi_{4} \nabla (v^{\sharp} \circ D_{\zeta}) \right) \circ D_{\zeta}^{-1}, \nu_{\partial M} \right\rangle.
\end{split}
\end{equation*}

\begin{rmk}
From Proposition \ref{prop.3}, if
$$
\left\lbrace
\begin{array}{l}
\mathcal{L}_{\varepsilon} \omega  = f,\;\;\text{in}\;\;M,\\
\mathcal{P}_{\varepsilon} \omega = h,\;\;\text{on}\;\;\partial M,
\end{array}
\right.
$$
then
\begin{equation} \label{eq.est}
\Vert \omega \Vert_{\mathcal{C}^{2,\alpha}_{\varepsilon}(M) \cap \mathcal{C}^{1,\alpha}_{\varepsilon}(\overline{M})} \leq C \left( \Vert f \Vert_{\mathcal{C}^{0,\alpha}_{\varepsilon}(M)} +  \Vert h  \Vert_{\mathcal{C}^{0,\alpha}_{\varepsilon}(\partial M)} \right).
\end{equation}
In the case where $f \equiv 0$ and $h \equiv 0$ in $\Omega_{4},$ we can show that the estimate for $\omega$ can be improved in $\Omega_{5}$ as
$$
\Vert \chi_{5} \omega \Vert_{\mathcal{C}^{2,\alpha}_{\varepsilon}(M) \cap \mathcal{C}^{1,\alpha}_{\varepsilon}(\overline{M})} \leq C \varepsilon \left( \Vert f \Vert_{\mathcal{C}^{0,\alpha}_{\varepsilon}(M)} +  \Vert h  \Vert_{\mathcal{C}^{0,\alpha}_{\varepsilon}(\partial M)} \right),
$$
provided that $\varepsilon$ is small enough. So, if $f \equiv 0$ and $h \equiv 0$ in $\Omega_{4},$ then we can improve \eqref{eq.est} into
\begin{equation} \label{eq.est1}
\Vert \omega \Vert_{\Tilde{\mathcal{C}}^{2,\alpha}_{\varepsilon}(M) \cap \Tilde{\mathcal{C}}^{1,\alpha}_{\varepsilon}(\overline{M})} \leq C \left( \Vert f \Vert_{\mathcal{C}^{0,\alpha}_{\varepsilon}(M)} +  \Vert h \Vert_{\mathcal{C}^{0,\alpha}_{\varepsilon}(\partial M)} \right),
\end{equation}
where, by definition
$$
\Vert v \Vert_{\Tilde{\mathcal{C}}^{2,\alpha}_{\varepsilon}(M) \cap \Tilde{\mathcal{C}}^{1,\alpha}_{\varepsilon}(\overline{M})} \coloneqq \varepsilon^{-2} \Vert \chi_{5} v \Vert_{\mathcal{C}^{2,\alpha}_{\varepsilon}(M) \cap \mathcal{C}^{1,\alpha}_{\varepsilon}(\overline{M})} + \Vert v \Vert_{\mathcal{C}^{2,\alpha}_{\varepsilon}(M) \cap \mathcal{C}^{1,\alpha}_{\varepsilon}(\overline{M})}
$$
\end{rmk}

Taking the difference between the equation satisfied by $v$ and the equation satisfied by $v^{\flat},$ we find that it is enough that $v^{\sharp}$ solves, in the support of $\chi_{4},$
\begin{equation*}
\left\lbrace
\begin{split}
\varepsilon^{2} \Delta_{g} (v^{\sharp} \circ D_{\zeta}) \circ D_{\zeta}^{-1} - W''(\Tilde{u}_{\varepsilon}) v^{\sharp} + \varepsilon \lambda &= E_{\varepsilon}(\zeta) + Q_{\varepsilon}(\chi_{4} v^{\sharp} + v^{\flat}) + \left( W''(\Tilde{u}_{\varepsilon}) -  \Gamma \right)v^{\flat} \\
&\;\;\;\;- \varepsilon^{2} \left( \Delta_{g} (v^{\flat} \circ D_{\zeta}) \circ D_{\zeta}^{-1} - \Delta_{g} v^{\flat} \right)\;\;\text{in}\;\;M,\\
\varepsilon \left\langle \nabla (v^{\sharp} \circ D_{\zeta}) \circ D_{\zeta}^{-1}, \nu_{\partial M} \right\rangle + \sigma''(\Tilde{u}_{\varepsilon}) v^{\sharp} &= - \varepsilon \left( \left\langle \nabla (v^{\flat} \circ D_{\zeta}) \circ D_{\zeta}^{-1},\nu_{\partial M} \right\rangle - \left\langle \nabla v^{\flat}, \nu_{\partial M} \right\rangle \right)\\
&\;\;\;\; + \Tilde{E}_{\varepsilon}(\zeta) + \Tilde{Q}_{\varepsilon}(\chi_{4} v^{\sharp} + v^{\flat}) - \left(\sigma''(\Tilde{u}_{\varepsilon}) - \Tilde{\Gamma} \right) v^{\flat}  \;\;\text{on}\;\;\partial M.
\end{split}
\right.
\end{equation*}

We can as well solve
\begin{equation} \label{eq.equiv}
\left\lbrace
\begin{split}
L_{\varepsilon} v^{\sharp} - \varepsilon J_{\Sigma} \zeta\; \dot{u}_{\varepsilon} &= M_{\varepsilon}(v^{\flat},v^{\sharp},\zeta),\;\;\text{in}\;\;M,\\
P_{\varepsilon} v^{\sharp} - \varepsilon B_{\Sigma} \zeta\; \dot{u}_{\varepsilon} &= V_{\varepsilon}(v^{\flat},v^{\sharp},\zeta),\;\;\text{on}\;\;\partial M,
\end{split}
\right.
\end{equation}
where
\begin{equation*}
\left\lbrace
\begin{split}
M_{\varepsilon}(v^{\flat},v^{\sharp},\zeta) \coloneqq \chi_{3}& \left[L_{\varepsilon} v^{\sharp} - \varepsilon^{2} \Delta_{g} (v^{\sharp} \circ D_{\zeta}) \circ D_{\zeta}^{-1} + W''(\Tilde{u}_{\varepsilon}) v^{\sharp} - \varepsilon \lambda + E_{\varepsilon}(\zeta) + Q_{\varepsilon}(\chi_{4} v^{\sharp} + v^{\flat})\right.\\
&\left.\;\;\;- \varepsilon J_{\Sigma} \zeta\; \dot{u}_{\varepsilon} - \varepsilon^{2} \left(\Delta_{g} (v^{\flat} \circ D_{\zeta}) \circ D_{\zeta}^{-1}  - \Delta_{g} v^{\flat}\right) + \left( W''(\Tilde{u}_{\varepsilon}) - \Gamma \right) v^{\flat} \right],\\
V_{\varepsilon}(v^{\flat},v^{\sharp},\zeta) \coloneqq \chi_{3}& \left[P_{\varepsilon} v^{\sharp} - \varepsilon \left\langle \nabla (v^{\sharp} \circ D_{\zeta}) \circ D_{\zeta}^{-1},\nu_{\partial M} \right\rangle - \sigma''(\Tilde{u}_{\varepsilon}) v^{\sharp} + \Tilde{E}_{\varepsilon}(\zeta) + \Tilde{Q}_{\varepsilon}(\chi_{4} v^{\sharp} + v^{\flat})\right.\\
&\left.\;\;\;- \varepsilon B_{\Sigma} \zeta\; \dot{u}_{\varepsilon} - \varepsilon \left( \left\langle \nabla (v^{\flat} \circ D_{\zeta}) \circ D_{\zeta}^{-1}, \nu_{\partial M} \right\rangle - \left\langle \nabla v^{\flat} ,\nu_{\partial M} \right\rangle \right) - \left( \sigma''(\Tilde{u}_{\varepsilon}) - \Tilde{\Gamma}  \right) v^{\flat} \right].
\end{split}
\right.
\end{equation*}

Let $\Pi$ be the orthogonal projection on $\dot{u}_{\varepsilon}$ defined in \eqref{eq.ortinha} and $\Pi^{\perp}$ be the orthogonal projection on the orthogonal of $\dot{u}_{\varepsilon},$ that is
$$
\Pi^{\perp}(f) \coloneqq f - \Pi(f) \dot{u}_{\varepsilon}.
$$

If $v^{\sharp}$ satisfies
$$
\int_{-\infty}^{\infty} v^{\sharp}(t,y) \dot{u}_{\varepsilon}(t)\;dt = 0\;\;\forall y \in \Sigma,
$$
then \eqref{eq.equiv} is equivalent to
\begin{equation*}
\left\lbrace
\begin{split}
L_{\varepsilon} v^{\sharp} &= \Pi^{\perp} \left[ M_{\varepsilon}(v^{\flat},v^{\sharp},\zeta)\right],\;\;\text{in}\;\;M,\\
- \varepsilon J_{\Sigma} \zeta\; \dot{u}_{\varepsilon} &= \Pi \left[ M_{\varepsilon}(v^{\flat},v^{\sharp},\zeta)\right],\;\;\text{in}\;\;M,\\
P_{\varepsilon} v^{\sharp} &= \Pi^{\perp} \left[V_{\varepsilon}(v^{\flat},v^{\sharp},\zeta)\right],\;\;\text{on}\;\;\partial M,\\
- \varepsilon B_{\Sigma} \zeta\; \dot{u}_{\varepsilon} &= \Pi\left[ V_{\varepsilon}(v^{\flat},v^{\sharp},\zeta)\right],\;\;\text{on}\;\;\partial M,\\
\int_{M} \Tilde{u}_{\varepsilon}\;dv_{g} &= c_{0}|M| - \int_{M} \left(\chi_{4} v^{\sharp} + v^{\flat}\right)\;dv_{g}.
\end{split}
\right.
\end{equation*}

\section{Proof of Theorem \ref{thm.main}} \label{sec.proof}
We look for a solution $u$ of \eqref{eq.AC} of the form
$$
u = \left( \Tilde{u}_{\varepsilon} + \chi_{4} v^{\sharp} + v^{\flat} \right) \circ D_{\zeta},
$$
where $v^{\sharp} \colon \mathbb{R} \times \Sigma \to \mathbb{R}$ is a function satisfying
$$
\int_{-\infty}^{\infty} v^{\sharp}(t,y) \dot{u}_{\varepsilon}(t)\;dt = 0\;\;\forall y \in \Sigma,
$$
and, together with the functions $v^{\flat} \colon M \to \mathbb{R}$ and $\zeta \colon \Sigma \to \mathbb{R},$ the system
\begin{equation} \label{eq.equiv1}
\left\lbrace
\begin{split}
\mathcal{L}_{\varepsilon} v^{\flat} &= N_{\varepsilon}(v^{\flat},v^{\sharp},\zeta),\;\;\text{in}\;\;M,\\
L_{\varepsilon} v^{\sharp} &= \Pi^{\perp} \left[ M_{\varepsilon}(v^{\flat},v^{\sharp},\zeta)\right],\;\;\text{in}\;\;M,\\
- \varepsilon J_{\Sigma} \zeta\; \dot{u}_{\varepsilon} &= \Pi \left[ M_{\varepsilon}(v^{\flat},v^{\sharp},\zeta)\right],\;\;\text{in}\;\;M,\\
\mathcal{P}_{\varepsilon} v^{\flat} &= \Tilde{N}_{\varepsilon}(v^{\flat},v^{\sharp},\zeta),\;\;\text{on}\;\;\partial M,\\
P_{\varepsilon} v^{\sharp} &= \Pi^{\perp} \left[V_{\varepsilon}(v^{\flat},v^{\sharp},\zeta)\right],\;\;\text{on}\;\;\partial M,\\
- \varepsilon B_{\Sigma} \zeta\; \dot{u}_{\varepsilon} &= \Pi\left[ V_{\varepsilon}(v^{\flat},v^{\sharp},\zeta)\right],\;\;\text{on}\;\;\partial M,\\
\int_{M} \Tilde{u}_{\varepsilon}\;dv_{g} - c_{0}|M| &=  - \int_{M} \left(\chi_{4} v^{\sharp} + v^{\flat}\right)\;dv_{g}.
\end{split}
\right.
\end{equation}

To deal with this problem as a fixed point one, we need the following estimates. Computations can be adapted from section $7$ of \cite{chodosh-mantoulidis}.

\begin{lem} \label{lem.8}
The following estimate holds
$$
\begin{array}{ll}
\Vert N_{\varepsilon}(0,0,0) \Vert_{\mathcal{C}^{0,\alpha}_{\varepsilon}(M)} & + \Vert \Pi^{\perp}(M_{\varepsilon}(0,0,0))\Vert_{\mathcal{C}^{0,\alpha}_{\varepsilon}(\mathbb{R}\times \Sigma)}\\
& + \Vert \Tilde{N}_{\varepsilon}(0,0,0) \Vert_{\mathcal{C}^{0,\alpha}_{\varepsilon}(\partial M)} + \Vert \Pi^{\perp}(V_{\varepsilon}(0,0,0))\Vert_{\mathcal{C}^{0,\alpha}_{\varepsilon}(\mathbb{R}\times \partial \Sigma)}  \leq C \varepsilon.
\end{array}
$$
Moreover
$$
\Vert \Pi(M_{\varepsilon}(0,0,0))\Vert_{\mathcal{C}^{0,\alpha}(\Sigma)} + \Vert \Pi(V_{\varepsilon}(0,0,0))\Vert_{\mathcal{C}^{0,\alpha}(\partial \Sigma)} \leq C \varepsilon.
$$
\end{lem}
\begin{proof}
Since
$$
\begin{array}{l}
N_{\varepsilon}(0,0,0) = \varepsilon \lambda - E_{\varepsilon}(0),\;\; M_{\varepsilon}(0,0,0) = -\varepsilon \lambda -E_{\varepsilon}(0),\\
\Tilde{N}_{\varepsilon}(0,0,0) = -\Tilde{E}_{\varepsilon}(0),\;\;  V_{\varepsilon}(0,0,0)) = \Tilde{E}_{\varepsilon}(0),
\end{array}
$$
and $\Tilde{u}_{\varepsilon} = \overline{u}_{\varepsilon}$ in the support of $\chi_{3}$ so that 
$$
E_{\varepsilon}(0) = -\varepsilon^{2} \Delta_{g} \Tilde{u}_{\varepsilon} + W'(\Tilde{u}_{\varepsilon}) = \varepsilon H_{t} \dot{u}_{\varepsilon} = \mathfrak{N}(0,0),
$$
and
$$
\Tilde{E}_{\varepsilon}(0) = -\varepsilon \left\langle \nabla \Tilde{u}_{\varepsilon}, \nu_{\partial M} \right\rangle - \sigma'(\Tilde{u}_{\varepsilon}) = \left(\left\langle \nu_{\Sigma_{t}}, \nu_{\partial M} \right\rangle - \cos \theta \right)\dot{u}_{\varepsilon} = \mathfrak{M}(0,0),
$$
the estimates follow from Lemma \ref{lem.4} and Lemma \ref{lem.5}.
\end{proof}

By computing directly and using the definition of the the $\mathcal{C}^{k,\alpha}_{\varepsilon}$ norms, we have the following.

\begin{lem} \label{lem.9}
Suppose that $v_{j}^{\flat},$ $v_{j}^{\sharp}$ and $\zeta_{j},$ $j = 1,2,$ are functions such that 
$$
\Vert v_{j}^{\flat} \Vert_{\Tilde{\mathcal{C}}^{2,\alpha}_{\varepsilon}(M) \cap \Tilde{\mathcal{C}}^{1,\alpha}_{\varepsilon}(\overline{M})} + \Vert v_{j}^{\sharp} \Vert_{\mathcal{C}^{2,\alpha}_{\varepsilon}(\mathbb{R}\times \Sigma) \cap \mathcal{C}^{1,\alpha}_{\varepsilon}(\mathbb{R}\times \overline{\Sigma})} + \varepsilon^{2\alpha} \Vert \zeta_{j} \Vert_{\mathcal{C}^{2,\alpha}(\Sigma) \cap \mathcal{C}^{1,\alpha}(\overline{\Sigma})} \leq \overline{C} \varepsilon,
$$
for some fixed $\overline{C} > 0,$ then there exists a constant $\delta > 0$ independent of $\alpha \in (0,1),$ such that
$$
\begin{array}{l}
\Vert N_{\varepsilon}(v_{2}^{\flat},v_{2}^{\sharp},\zeta_{2}) - N_{\varepsilon}(v_{1}^{\flat},v_{1}^{\sharp},\zeta_{1})\Vert_{\mathcal{C}^{0,\alpha}_{\varepsilon}(M)} + \Vert \Tilde{N}_{\varepsilon}(v_{2}^{\flat},v_{2}^{\sharp},\zeta_{2}) - \Tilde{N}_{\varepsilon}(v_{1}^{\flat},v_{1}^{\sharp},\zeta_{1}) \Vert_{\mathcal{C}^{0,\alpha}_{\varepsilon}(\partial M)} \\
\quad \leq C \varepsilon^{\delta} \left(\Vert v_{2}^{\flat} - v_{1}^{\flat} \Vert_{\mathcal{C}^{2,\alpha}_{\varepsilon}(M) \cap \mathcal{C}^{1,\alpha}_{\varepsilon}(\overline{M})} + \Vert v_{2}^{\sharp} - v_{1}^{\sharp} \Vert_{\mathcal{C}^{2,\alpha}_{\varepsilon}(\mathbb{R}\times \Sigma) \cap \mathcal{C}^{1,\alpha}_{\varepsilon}(\mathbb{R}\times \overline{\Sigma})}  + \Vert \zeta_{2} - \zeta_{1} \Vert_{\mathcal{C}^{2,\alpha}(\Sigma) \cap \mathcal{C}^{1,\alpha}(\overline{\Sigma})} \right),
\end{array}
$$
$$
\begin{array}{l}
\Vert \Pi^{\perp}(M_{\varepsilon}(v_{2}^{\flat},v_{2}^{\sharp},\zeta_{2}) - M_{\varepsilon}(v_{1}^{\flat},v_{1}^{\sharp},\zeta_{1})) \Vert_{\mathcal{C}^{0,\alpha}(\mathbb{R} \times \Sigma)} + \Vert \Pi^{\perp}(V_{\varepsilon}(v_{2}^{\flat},v_{2}^{\sharp},\zeta_{2}) - V_{\varepsilon}(v_{1}^{\flat},v_{1}^{\sharp},\zeta_{1})) \Vert_{\mathcal{C}^{0,\alpha}(\mathbb{R} \times \partial \Sigma)} \\
\quad \leq C \varepsilon^{\delta} \left(\Vert v_{2}^{\flat} - v_{1}^{\flat} \Vert_{\Tilde{\mathcal{C}}^{2,\alpha}_{\varepsilon}(M) \cap \Tilde{\mathcal{C}}^{1,\alpha}_{\varepsilon}(\overline{M})} + \Vert v_{2}^{\sharp} - v_{1}^{\sharp} \Vert_{\mathcal{C}^{2,\alpha}_{\varepsilon}(\mathbb{R}\times \Sigma) \cap \mathcal{C}^{1,\alpha}_{\varepsilon}(\mathbb{R}\times \overline{\Sigma})}  + \Vert \zeta_{2} - \zeta_{1} \Vert_{\mathcal{C}^{2,\alpha}(\Sigma)\cap \mathcal{C}^{1,\alpha}(\overline{\Sigma})} \right),
\end{array}
$$
and
$$
\begin{array}{l}
\Vert \Pi(M_{\varepsilon}(v_{2}^{\flat},v_{2}^{\sharp},\zeta_{2}) - M_{\varepsilon}(v_{1}^{\flat},v_{1}^{\sharp},\zeta_{1})) \Vert_{\mathcal{C}^{0,\alpha}(\Sigma)} + \Vert \Pi(V_{\varepsilon}(v_{2}^{\flat},v_{2}^{\sharp},\zeta_{2}) - V_{\varepsilon}(v_{1}^{\flat},v_{1}^{\sharp},\zeta_{1})) \Vert_{\mathcal{C}^{0,\alpha}(\partial \Sigma)}\\
\leq C \varepsilon^{1 - \alpha} \Vert v_{2}^{\flat} - v_{1}^{\flat} \Vert_{\mathcal{C}^{2,\alpha}_{\varepsilon}(M) \cap \mathcal{C}^{1,\alpha}_{\varepsilon}(\overline{M})} + \varepsilon^{1 + \delta} \left( \Vert v_{2}^{\sharp} - v_{1}^{\sharp} \Vert_{\Tilde{\mathcal{C}}^{2,\alpha}_{\varepsilon}(\mathbb{R}\times \Sigma) \cap \Tilde{\mathcal{C}}^{1,\alpha}_{\varepsilon}(\mathbb{R}\times \overline{\Sigma})}  + \Vert \zeta_{2} - \zeta_{1} \Vert_{\mathcal{C}^{2,\alpha}(\Sigma) \cap \mathcal{C}^{1,\alpha}(\overline{\Sigma})} \right).
\end{array}
$$
\end{lem}


Choose $\alpha \in (0,1)$ close to $0,$ and assume that we are given $v^{\flat} \in \Tilde{\mathcal{C}}^{2,\alpha}_{\varepsilon}(M) \cap \Tilde{\mathcal{C}}^{1,\alpha}_{\varepsilon}(\overline{M}),$ $v^{\sharp} \in \mathcal{C}^{2,\alpha}_{\varepsilon}(\mathbb{R}\times \Sigma) \cap \mathcal{C}^{1,\alpha}_{\varepsilon}(\mathbb{R}\times \overline{\Sigma})$ and $\zeta \in \mathcal{C}^{2,\alpha}(\Sigma) \cap \mathcal{C}^{1,\alpha}(\overline{\Sigma})$ such that
\begin{equation} \label{eq.condition}
\Vert v^{\flat} \Vert_{\Tilde{\mathcal{C}}^{2,\alpha}_{\varepsilon}(M) \cap \Tilde{\mathcal{C}}^{1,\alpha}_{\varepsilon}(\overline{M})} + \Vert v^{\sharp} \Vert_{\mathcal{C}^{2,\alpha}_{\varepsilon}(\mathbb{R}\times \Sigma) \cap \mathcal{C}^{1,\alpha}_{\varepsilon}(\mathbb{R}\times \overline{\Sigma})} + \varepsilon^{2\alpha} \Vert \zeta \Vert_{\mathcal{C}^{2,\alpha}(\Sigma) \cap \mathcal{C}^{1,\alpha}(\overline{\Sigma})} \leq \Tilde{C} \varepsilon,
\end{equation}
for some constant $\Tilde{C}.$ 

Using Lemma \ref{lem.8}, Proposition \ref{prop.2}, Proposition \ref{prop.3} and the nondegeneracy condition, we find $\Tilde{v}^{\flat} \in \Tilde{\mathcal{C}}^{2,\alpha}_{\varepsilon}(M) \cap \Tilde{\mathcal{C}}^{1,\alpha}_{\varepsilon}(\overline{M}),$ $\Tilde{v}^{\sharp} \in \mathcal{C}^{2,\alpha}_{\varepsilon}(\mathbb{R}\times \Sigma) \cap  \mathcal{C}^{1,\alpha}_{\varepsilon}(\mathbb{R}\times \overline{\Sigma})$ and $\Tilde{\zeta} \in \mathcal{C}^{2,\alpha}(\Sigma) \cap \mathcal{C}^{1,\alpha}(\overline{\Sigma})$  such that
\begin{equation} \label{eq.fixed1}
\left\lbrace
\begin{split}
\mathcal{L}_{\varepsilon} \Tilde{v}^{\flat} &= N_{\varepsilon}(v^{\flat},v^{\sharp},\zeta),\;\;\text{in}\;\;M,\\
L_{\varepsilon} \Tilde{v}^{\sharp} &= \Pi^{\perp} \left[ M_{\varepsilon}(v^{\flat},v^{\sharp},\zeta)\right],\;\;\text{in}\;\;M,\\
- \varepsilon J_{\Sigma} \Tilde{\zeta}\; \dot{u}_{\varepsilon} &= \Pi \left[ M_{\varepsilon}(v^{\flat},v^{\sharp},\zeta)\right],\;\;\text{in}\;\;M,\\
\mathcal{P}_{\varepsilon} \Tilde{v}^{\flat} &= \Tilde{N}_{\varepsilon}(v^{\flat},v^{\sharp},\zeta),\;\;\text{on}\;\;\partial M,\\
P_{\varepsilon} \Tilde{v}^{\sharp} &= \Pi^{\perp} \left[V_{\varepsilon}(v^{\flat},v^{\sharp},\zeta)\right],\;\;\text{on}\;\;\partial M,\\
- \varepsilon B_{\Sigma} \Tilde{\zeta}\; \dot{u}_{\varepsilon} &= \Pi\left[ V_{\varepsilon}(v^{\flat},v^{\sharp},\zeta)\right],\;\;\text{on}\;\;\partial M,
\end{split}
\right.
\end{equation}
and
$$
\Vert \Tilde{v}^{\flat} \Vert_{\Tilde{\mathcal{C}}^{2,\alpha}_{\varepsilon}(M) \cap \Tilde{\mathcal{C}}^{1,\alpha}_{\varepsilon}(\overline{M})} + \Vert \Tilde{v}^{\sharp} \Vert_{\mathcal{C}^{2,\alpha}_{\varepsilon}(\mathbb{R}\times \Sigma) \cap \mathcal{C}^{1,\alpha}_{\varepsilon}(\mathbb{R}\times \overline{\Sigma})} + \varepsilon^{2\alpha} \Vert \Tilde{\zeta} \Vert_{\mathcal{C}^{2,\alpha}(\Sigma) \cap \mathcal{C}^{1,\alpha}(\overline{\Sigma})} \leq \Tilde{C} \varepsilon,
$$
for some constant $\Tilde{C}.$ Thus, we have a continuous mapping from $\Tilde{\mathcal{C}}^{2,\alpha}_{\varepsilon}(M) \cap \Tilde{\mathcal{C}}^{1,\alpha}_{\varepsilon}(\overline{M}) \times  \mathcal{C}^{2,\alpha}_{\varepsilon}(\mathbb{R}\times \Sigma) \cap \mathcal{C}^{1,\alpha}_{\varepsilon}(\mathbb{R}\times \overline{\Sigma}) \times \mathcal{C}^{2,\alpha}(\Sigma) \cap \mathcal{C}^{1,\alpha}(\overline{\Sigma})$  satisfying \eqref{eq.condition} into itself. By Lemma \ref{lem.9} together with the contractivity of $\int_{M} \left(\chi_{4} \Tilde{v}^{\sharp} + \Tilde{v}^{\flat}\right)\;dv_{g}$ to include the mass constraint, we can apply the fixed point theorem for contraction mappings to find the existence of a unique solution $(v^{\flat},v^{\sharp},\zeta)$ of \eqref{eq.equiv1} satisfying
$$
\Vert v^{\flat} \Vert_{\Tilde{\mathcal{C}}^{2,\alpha}_{\varepsilon}(M) \cap \Tilde{\mathcal{C}}^{1,\alpha}_{\varepsilon}(\overline{M})} + \Vert v^{\sharp} \Vert_{\mathcal{C}^{2,\alpha}_{\varepsilon}(\mathbb{R}\times \Sigma) \cap \mathcal{C}^{1,\alpha}_{\varepsilon}(\mathbb{R}\times \overline{\Sigma})} + \varepsilon^{2\alpha} \Vert \zeta \Vert_{\mathcal{C}^{2,\alpha}(\Sigma) \cap \mathcal{C}^{1,\alpha}(\overline{\Sigma})} \leq \overline{C} \varepsilon,
$$
for some $\overline{C} > 0$ large enough. We have found then a solution $u = (\overline{u}_{\varepsilon} + \chi_{4} v^{\sharp} + v^{\flat}) \circ D_{\zeta}$ of \eqref{eq.AC} and the proof is concluded.

\end{document}